\documentclass[11pt]{article}
\usepackage{amsmath, amssymb, amsthm}
\usepackage{verbatim}
\usepackage{multicol}
\usepackage{enumerate}
\usepackage{comment}
\usepackage[none]{hyphenat}
\usepackage{hyperref}
\hypersetup{
	colorlinks=true,
	linkcolor=blue,
	filecolor=magenta,
	urlcolor=cyan,
	citecolor=blue
}

\usepackage{pgf}
\usepackage{tikz}
\usetikzlibrary{math,calc,intersections}
\usetikzlibrary{positioning,arrows,shapes,decorations.markings,decorations.pathreplacing,matrix,patterns}
\tikzstyle{vertex}=[circle,draw=black,fill=black,inner sep=0,minimum size=3pt,text=white,font=\footnotesize]
\usepackage{cleveref}

\date{}
\title{\vspace{-0.8cm}Hasse diagrams with large chromatic number}
\author{
	Andrew Suk\thanks{UC San Diego.  \emph{e-mail}: \textbf{asuk@ucsd.edu}, Research supported by an NSF CAREER award and an Alfred Sloan Fellowship.}
	\and
	Istv\'{a}n Tomon\thanks{ETH Zurich. \emph{e-mail}: \textbf{istvan.tomon@math.ethz.ch}, Research supported by SNSF grant 200021-149111.}
}

\theoremstyle{plain}
\newtheorem{theorem}{Theorem}

\newtheorem{claim}[theorem]{Claim}
\newtheorem{lemma}[theorem]{Lemma}
\newtheorem{conjecture}[theorem]{Conjecture}
\newtheorem{problem}[theorem]{Problem}

\theoremstyle{definition}

\newcommand{\eps}{\varepsilon}

\begin{document}
	\maketitle
	
	\sloppy
	
	\begin{abstract}
		For every positive integer $n$, we construct a Hasse diagram with $n$ vertices and chromatic number $\Omega(n^{1/4})$, which significantly improves on the previously known best constructions of Hasse diagrams having chromatic number $\Theta(\log n)$. In addition, if we also require that our Hasse diagram has girth at least $k\geq 5$, we can achieve a chromatic number of at least $n^{\frac{1}{2k-3}+o(1)}$. %Our constructions are also semi-algebraic, which might be of independent interest.
		
		These results have the following surprising geometric consequence. They imply the existence of a family $\mathcal{C}$ of $n$ curves in the plane such that the disjointness graph $G$ of $\mathcal{C}$ is triangle-free (or have high girth), but the chromatic number of $G$ is polynomial in $n$. Again, the previously known best construction, due to Pach, Tardos and T\'oth, had only logarithmic chromatic number.
	\end{abstract}

	\section{Introduction}
	Let $G$ be a graph.  The independence number of $G$ is denoted by $\alpha(G)$, the clique number of $G$ is $\omega(G)$, and the chromatic number of $G$ is $\chi(G)$. Also, the  \emph{girth} of $G$ is the length of its shortest cycle.
	
	Given a triangle-free graph on $n$ vertices, how large can its chromatic number be? This question is closely related to the problem of determining the asymmetric Ramsey number $R(3,m)$, which is the maximum number of vertices in a triangle-free graph with no independent set of size $m$. It was proved by Ajtai, Koml\'os and Szemer\'edi \cite{AKSz80} that $R(3,m)=O(\frac{m^{2}}{\log m})$, and Kim \cite{K95} proved a matching lower bound with the help of probabilistic tools.  Finding explicit constructions of triangle-free graphs with no independent set of size $m$ has also received a lot of attention over the past several decades. The best known explicit constructions of such graphs have $\Theta(m^{3/2})$ vertices, see \cite{A94,CPR00,KPR10}. Moreover, explicit constructions for other asymmetric Ramsey numbers $R(s,m)$, where $s$ is viewed as a fixed constant, are considered by Alon and Pudl\'ak in \cite{AP01}.
	
	The aforementioned results imply that the chromatic number of a triangle-free graph on $n$ vertices is at most $n^{1/2+o(1)}$, and this bound is best possible. However, the question whether this can be improved for specific families of triangle-free graphs has been extensively studied, and one such family of particular interest is the family of Hasse diagrams.
	
	Hasse diagrams were introduced by Vogt~\cite{V95} at the end of the 19th century for concise representation of partial orders. Today, they are widely used in graph drawing algorithms. Given a partially ordered set $(P,<)$, its \emph{Hasse diagram} (or \emph{cover graph})  is the graph on vertex set $P$ in which $x$ and $y$ are joined by an edge if $x<y$ and there exists no $z\in P$ such that $x<z<y$. It is easy to see that if $G$ is a Hasse diagram, then $G$ is triangle-free. It is already not trivial that the chromatic number of Hasse diagrams can be arbitrarily large, but the following construction of Erd\H{o}s and Hajnal \cite{ErH64} shows just that.  The \emph{shift graph of order $N$} is the graph $G$ with vertex set $\{(i,j):1\leq i<j\leq N\}$ in which $(i,j)$ and $(i',j')$ are joined by an edge if $j=i'$. Then $G$ is the Hasse diagram of the poset defined as $(i,j)\leq (i',j')$ if $j\leq i'$; also it is a nice exercise to show that $\chi(G)=\lfloor\log_{2}N\rfloor$. Bollob\'as \cite{B77} proved that there exists Hasse diagrams (in particular, Hasse diagrams of lattices) with arbitrarily large girth  and chromatic number. Ne\v set\v ril and R\"odl \cite{NR79} gave an alternative proof of this result.  The construction of Bollob\'as, when optimized, implies the existence of Hasse diagrams with $n$ vertices, girth at least $k$ (where $k$ is any fixed integer) and chromatic number $\Theta(\frac{\log n}{\log\log n})$. Recently, under the same conditions, the bound on the chromatic number was improved to $\Omega(\log n)$ by Pach and Tomon \cite{PT19}, and they showed that this bound is optimal for the Hasse diagrams of so called uniquely generated posets. In \cite{KN91}, it is proved that there exists Hasse diagrams of 2-dimensional posets with arbitrarily large chromatic number, and in \cite{CPST09}, it is proved that the chromatic number of Hasse diagrams of 2-dimensional posets can be as large as $\Omega(\frac{\log n}{(\log\log n)^{2}})$, where $n$ is the number of vertices. Felsner et al. \cite{FGMR94} considered Hasse diagrams of interval orders and $N$-free posets of height $h$, and proved that the maximum chromatic number of such Hasse diagrams is logarithmic in $h$.
	
	All of the aforementioned constructions of $n$ vertex Hasse diagrams have chromatic number $O(\log n)$. Therefore, one might suspect that the  chromatic number of Hasse diagrams is at most logarithmic.  Our main result shows that this intuition is false, by constructing Hasse diagrams whose chromatic number is polynomial in the number of vertices. 
	
	\begin{theorem}\label{thm:mainthm}
		For every positive integer $n$, there exists a Hasse diagram with $n$ vertices and chromatic number $\Omega(n^{1/4})$.
	\end{theorem}
	
	Unlike in the triangle-free graph case, probabilistic tools are not suitable to construct Hasse diagrams with large chromatic number. Indeed, not only do Hasse diagrams avoid triangles, but they avoid an infinite family of ordered cycles (see Claim \ref{claim:char} for further details). However, surprisingly, the explicit constructions of Ramsey-graphs in \cite{CPR00,KPR10} are almost tailor-made for this problem. Roughly, in these constructions, they consider the so-called super-line graph of point-line incidence graphs over finite planes. Here, we show that if one considers the super-line graph of point-line incidence graphs over the real plane instead, one also gets a Hasse diagram.
	
	Furthermore, we prove that even if one requires a Hasse diagram with large girth, we can achieve polynomial chromatic number. Indeed, we get this almost immediately by randomly sparsifying our construction from Theorem \ref{thm:mainthm}. 
	
	\begin{theorem}\label{thm:girth}
		For every positive integer $k\geq 5$, there exists a Hasse diagram on $n$ vertices with girth at least $k$ and chromatic number at least $n^{\frac{1}{2k-3}+o(1)}$. 
	\end{theorem}
	
	Let us remark that our constructions are also semi-algebraic with bounded complexity, which might be of independent interest.
	
	\subsection{Coloring geometric graphs}
	
	A family of graphs $\mathcal{G}$ is \emph{$\chi$-bounded} if there exists a function $f:\mathbb{N}\rightarrow \mathbb{N}$ such that $\chi(G)\leq f(\omega(G))$ for every $G\in \mathcal{G}$. Such a function $f$ is called \emph{$\chi$-bounding} for $\mathcal{G}$. The \emph{intersection graph} of a family of geometric objects $\mathcal{C}$ is the graph whose vertices correspond to the elements of $\mathcal{C}$, and two vertices are joined by an edge if the corresponding sets have a nonempty intersection. Also, the \emph{disjointness graph} of $\mathcal{C}$ is the complement of the intersection graph. A \emph{curve} is the image of a continuous function $\phi:[0,1]\rightarrow \mathbb{R}^{2}$, and a curve is \emph{$x$-monotone}, if every vertical line intersects it in at most 1 point.
	
	Results about the $\chi$-boundedness of intersection and disjointness graphs of geometric objects have a vast literature. One of the first such results is due to Asplund and Gr\"unbaum \cite{AG60}, who proved that the family of intersection graphs of axis-parallel rectangles in the plane is $\chi$-bounded. In \cite{Gy85,KK97}, it is proved that intersection graphs of chords of a circle are $\chi$-bounded, and recently  Davies and McCarty \cite{DM19} proved that $f(x)=O(x^{2})$ is $\chi$-bounding for this family. However, the family of intersection graphs of segments is not $\chi$-bounded: Pawlik et al.~\cite{PaKK14} gave a construction of $n$ segments in the plane whose intersection graph is triangle-free and has chromatic number $\Omega(\log\log n)$. On the other hand, it follows from a result of Rok and Walczak \cite{RW19} that $\chi(G)=O_{w(G)}(\log n)$ if $G$ is the intersection graph of $n$ segments, or more generally $x$-monotone curves, and Fox and Pach \cite{FP14} showed that $\chi(G)=(\log n)^{O(\log w(G))}$ if $G$ is the intersection graph of curves. In contrast, if we assume that $G$ is the intersection graph of curves, and the girth of $G$ is at least $5$, then $\chi(G)$ is bounded by a constant \cite{FP14}. 
	
	In the case of disjointness graphs, it follows from an argument of Larman, Matou\v{s}ek, Pach, and  T\"or\H{o}csik \cite{LMPT94} that the family of disjointness graphs of convex sets and $x$-monotone curves is $\chi$-bounded with $\chi$-bounding function $f(x)=x^{4}$. Recently, Pach and Tomon \cite{PT20} proved that in the case of $x$-monotone curves, this $\chi$-bounding function is optimal up to a constant factor. However, the family of disjointness graphs of curves is not $\chi$-bounded, Pach, Tardos and T\'oth \cite{PaTT17} constructed a family of $n$ curves, each being a polygonal line of 4 segments, whose disjointness graph is triangle-free and has chromatic number $\Theta(\log n)$. After the aforementioned results, it seems reasonable to conjecture that the chromatic number of triangle-free (or large girth) disjointness graphs of curves is also at most poly-logarithmic. Surprisingly, this is completely false.
	
	It turns out that Hasse diagrams and disjointness graphs are closely related. A curve is \emph{grounded} if it is contained in the nonnegative half-plane, and one of its endpoints lies on the $y$-axis. Middendorf and Pfeiffer \cite{MiP93} proved that $G$ is a Hasse-diagram if and only if there exists a family of grounded curves whose disjointness graph is triangle-free and isomorphic to $G$. But then Theorem \ref{thm:mainthm} and Theorem \ref{thm:girth} immediately imply the following result. 
	
	\begin{theorem}
		For every positive integer $n$, there exists a family of $n$ curves, whose disjointness graph $G$ is triangle-free and $\chi(G)\geq \Omega(n^{1/4})$. Also, for every $k\geq 5$, there exists a family of $n$ curves, whose disjointness graph $G$ has girth $k$ and $\chi(G)\geq n^{\frac{1}{2k-3}+o(1)}$.
	\end{theorem} 
	
	Hasse diagrams also appear in connection to another geometric problem. Given a point set $P$ in the plane, its \emph{Delaunay graph} with respect to axis-parallel rectangles is the graph whose vertices are $P$, and $p$ and $q$ are joined by an edge if there exists an axis-parallel rectangle whose intersection with $P$ is exactly $\{p,q\}$. Even et al. \cite{ELRS03} asked whether there exists a constant $c>0$ such that the Delaunay graph of any set of $n$ points contains an independent set of size $cn$. Chen et al. \cite{CPST09} proved that the answer is no by showing that a random set of $n$ points has no independent set of size larger than $O(n\frac{(\log\log n)^{2}}{\log n})$. On the other hand, it remains open whether every such Delaunay graph contains an independent set of size $n^{1-o(1)}$. However, given a set of points $P$, one can consider the partial ordering $<$ on $P$, where $(a,b)<(c,d)$ if $a<c$ and $b<d$. Then $(P,<)$ is a $2$-dimensional poset (in particular, every $2$-dimensional poset can be defined this way). If $G$ is the Hasse diagram of $(P,<)$, then $G$ is a subgraph (on the same vertex set) of the Delaunay graph of $P$. Therefore, if one could construct a 2-dimensional poset with $n$ vertices whose Hasse diagram has independence number at most $n^{1-\epsilon}$, it would imply the existence of an $n$-element point set in the plane whose Delaunay graph has no independent set larger than $n^{1-\epsilon}$.
	
	\begin{problem}
		Are there 2-dimensional posets on $n$ vertices whose Hasse diagram has independence number $n^{1-\eps}$?

	\end{problem}
	
	Our construction for Theorem \ref{thm:mainthm} has independence number $O(n^{3/4})$, which gives some hope that there might exist such 2-dimensional posets, but so far we were unable to construct one.
	
	\section{Constructing Hasse diagrams with large chromatic number}\label{sect:1}
	
	In this section, we outline our general construction and prove Theorem \ref{thm:mainthm}.
	
	An ordered graph is a pair $(G,\prec)$, where $G$ is a graph and $\prec$ is a total order on $V(G)$. A \emph{monotone cycle} of length $k$ in an ordered graph $(G,\prec)$ is a $k$-element subset of the vertices $v_{1}\prec\dots\prec v_{k}$ such that $v_{1}v_{k}$ and $v_{i}v_{i+1}$ are edges of $G$ for $i=1,\dots,k-1$. The following simple characterization of Hasse diagrams will come in handy.
	
	\begin{claim}\label{claim:char}
		$G$ is a Hasse diagram if and only if there exists an ordering $\prec$ on $V(G)$ such that $(G,\prec)$ does not contain a monotone cycle.
	\end{claim}
	
	\begin{proof}
		Let $G$ be a graph and let  $\prec$ be an ordering on $V(G)$ such that $(G,\prec)$ does not contain a monotone cycle. Define the relation $<$ on $V(G)$ such that $x<y$ if there exists a sequence of vertices $x=v_{1}\prec v_{2}\prec \dots\prec v_{m}=y$ such that $v_{i}v_{i+1}$ is an edge for $i=1,\dots,m-1$. Then it is easy to check that $<$ is a partial ordering on $V(G)$, and if $(G,\prec)$ has no monotone cycles, then the Hasse diagram of $(V(G),<)$ is $G$.
		
		On the other hand, if $G$ is a Hasse diagram of some poset $(V(G),<)$, then let $\prec$ be any \emph{linear extension} of $<$, that is, a total order which satisfies $x\prec y$ if $x<y$. Then $(G,\prec)$ does not contain a monotone cycle.
	\end{proof}

	Let $P$ be a set of $N$ points, and $\mathcal{L}$ be a set of $N$ lines on the plane, where we specify $N$ later. We denote by $I(P,\mathcal{L})$ the set of incidences between $P$ and $\mathcal{L}$, that is, the set of pairs $(p,l)\in P\times \mathcal{L}$ such that $p\in l$. Also, let $x(p)$ denote the $x$-coordinate of a point $p$, and let $s(l)$ denote the slope of a line $l$. By applying a projection to $(P,\mathcal{L})$, we can assume that the points have different $x$-coordinates, and the lines have different slopes, without changing the incidence structure. 
	
	Define the ordered graph $(G,\prec)$ as follows. Let $I(P,\mathcal{L})$ be the vertex set of $G$, and define the ordering $\prec$ such that $(p,l)\prec (p',l')$ if $x(p)< x(p')$, or $p=p'$ and $s(l)< s(l')$. Now define the edge set of $G$ as follows. Join $(p,l)$ and $(p',l')$ by an edge if $x(p)< x(p')$, $s(l)<s(l')$ and $p'\in l$. The key observation is the following.
	
	\begin{claim}
		$G$ is a Hasse diagram.
	\end{claim}
	
	\begin{proof}
		By Claim \ref{claim:char}, it is enough to show that $(G,\prec)$ does not contain a monotone cycle. Indeed, suppose that $(p_{1},l_{1}),\dots,(p_{k},l_{k})$ are the vertices of a monotone cycle. Then
		\begin{itemize}
			\item $x(p_{1})<\dots< x(p_{k})$,
			\item $s(l_{1})<\dots< s(l_{k})$, 
			\item $p_{1}\in l_{1}$, and  $l_{i-1}\cap l_{i}=\{p_{i}\}$ for $i=2,\dots,k$.
		\end{itemize}   
		
		These properties imply that $p_{1},\dots,p_{k}$ and $l_{1},\dots,l_{k}$ form a convex sequence, see Figure \ref{figure:path} for an illustration.  But then $p_{k}$ is strictly above the line $l_{1}$, so $p_{k}\not\in l_{1}$, contradicting that $(p_{1},l_{1})$ and $(p_{k},l_{k})$ are joined by an edge.  
	\end{proof}
	
	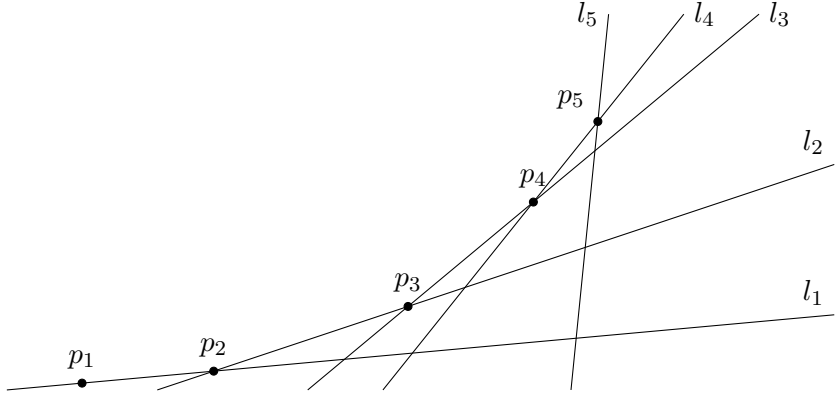
\begin{figure}
		\centering
		\begin{tikzpicture}
		
		%\node[vertex] at (0,0) {} ; \node at (0,-0.3) {$p_{1}$} ; 
		
		\coordinate (A) at (-1,0); \coordinate[label=above left:$l_{1}$] (A1) at (10,1); \draw (A) -- (A1) ;
		\coordinate (B) at (1,0);  \coordinate[label=above left:$l_{2}$] (B1) at (10,3); \draw (B) -- (B1) ;
		\coordinate (C) at (3,0);  \coordinate[label=right:$l_{3}$] (C1) at (9,5); \draw (C) -- (C1) ;
		\coordinate (D) at (4,0);  \coordinate[label=right:$l_{4}$] (D1) at (8,5); \draw (D) -- (D1) ;
		\coordinate (E) at (6.5,0); \coordinate[label=left:$l_{5}$] (E1) at (7,5); \draw (E) -- (E1) ;
		
		\node[vertex,label=above:$p_{1}$] (p1) at (0,1/11) {};
		\node[vertex,label=above:$p_{2}$] (p2) at (intersection of A--A1 and B--B1) {};
		\node[vertex,label=above:$p_{3}$] (p3) at (intersection of B--B1 and C--C1) {};
		\node[vertex,label=above:$p_{4}$] (p4) at (intersection of C--C1 and D--D1) {};
		\node[vertex,label=above left:$p_{5}$] (p5) at (intersection of D--D1 and E--E1) {};

		\end{tikzpicture}
		\caption{A monotone path $(p_{1},l_{1}),\dots,(p_{5},l_{5})$.}
		\label{figure:path}
	\end{figure}
	
	Next, we show that $G$ contains no large independent sets. 
	
	\begin{claim}\label{claim:ind}
		$G$ has no independent set larger than $2N$.
	\end{claim}
	
	\begin{proof}
		Let $B$ be the incidence graph of $(P,\mathcal{L})$, that is, $B$ is the bipartite graph with parts $P$ and $\mathcal{L}$, and $I(P,\mathcal{L})$ is the edge set of $B$. Also, define the ordering $\prec$ on $V(B)$ such that for $p,p'\in P$, $p\prec p'$ if $x(p)<x(p')$, for $l,l'\in\mathcal{L}$, $l\prec l'$ if $s(l)<s(l')$, and $p\prec l$ for every $p\in P$, $l\in \mathcal{L}$. If $J\subset I(P,\mathcal{L})$ is an independent set of $G$, then there exists no $p,p'\in P$ and $l,l'\in \mathcal{L}$ such that $x(p)<x(p')$, $s(l)<s(l')$ and $(p,l),(p',l),(p',l')\in J$ (note that this is actually a stronger condition than $J$ being an independent set). But this configuration of 3 edges in $B$ is equivalent to a copy of the ordered path $S$, which has 4 vertices $x_{1}\prec x_{2}\prec x_{3}\prec x_{4}$ and 3 edges $x_{1}x_{3},x_{2}x_{3},x_{2}x_{4}$.  
		
		Therefore, it is enough to show that if $(H,\prec)$ is a bipartite ordered graph with parts $X\prec Y$ of size $N$, and $H$ does not contain a copy of the ordered path $S$, then $H$ has at most $2N$ edges. This statement is well known in different forms, but due to its simplicity, let us present the proof.
		Suppose that $H$ has at least $2N+1$ edges. For each $x\in X$, let $y_{x}$ denote the largest element of $Y$ (with respect to $\prec$) such that $(x,y_{x})$ is an edge of $H$, and if $x$ is an isolated vertex, do not define $y_{x}$. For each $x\in X$, delete the edge $(x,y_{x})$ from $H$, and let $H'$ be the resulting graph. Then $H'$ still has at least $N+1$ edges. Therefore, there exists $y\in Y$ of degree at least 2, let $x\prec x'$ be two neighbors of $y$. But then setting $y'=y_{x'}$, the edges $xy, x'y,x'y'$ form a copy of $S$. 
	\end{proof}
	
	\noindent Therefore, by using the inequality $\chi(G)\geq \frac{|V(G)|}{\alpha(G)}$, we get $\chi(G)\geq \frac{|I(P,\mathcal{L})|}{2N}$. 
	
	We finish the proof by noting the well known result that there exists a point-line configuration $(P,\mathcal{L})$ such that $|P|=|\mathcal{L}|=N$ and $|I(P,\mathcal{L})|=\Theta(N^{4/3})$, which is the maximum number of incidences by the celebrated Szemer\'edi-Trotter theorem \cite{SzT83}. Indeed, setting $n=|I(P,\mathcal{L})|$, such a configuration gives a Hasse diagram $G$ with $n$ vertices and chromatic number $\chi(G)=\Omega(N^{1/3})=\Omega(n^{1/4})$.
	
	As we will also need it later, let us describe a point-line configuration with $\Theta(N^{4/3})$ incidences, which we shall refer to as the \emph{standard configuration}. Take $$P=\{(a,b)\in \mathbb{N}^{2}:a<N^{1/3}, b<N^{2/3}\}$$ and $$\mathcal{L}=\{ax+b=y:a,b\in \mathbb{N},a<N^{1/3},b<N^{2/3}\}.$$
	\hfill$\Box$
	
	\section{Large girth}
	
	In this section, we prove Theorem \ref{thm:girth}. We will consider the same graph $G$ as in the previous section, but in order to guarantee high girth, we chose our point-line configuration $(P,\mathcal{L})$ differently.
	
	Let $(P,\mathcal{L})$ be a point-line configuration, and let $B$ be the incidence graph of $(P,\mathcal{L})$. 
	
	\begin{claim}\label{claim:girth}
		If the girth of $B$ is at least $2k-2$, then the girth of $G$ is at least $k$.
	\end{claim}
	
	\begin{proof}
		Suppose that $G$ contains a cycle of length $r$ for some $r\leq k-1$, and let the vertices of such a cycle be $(p_{1},l_{1}),\dots,(p_{r},l_{r})$. Then $l_{1},p_{2},l_{2},\dots,p_{r-1},l_{r-1},p_{r},l_{1}$ is a closed walk of length $2r-2$ in $B$. But no two consecutive edges of this walk are the same, so this walk contains a cycle of length at most $2r-2\leq 2k-4$ in $B$, contradiction.
	\end{proof}
	
	Therefore, our task is reduced to constructing a point-line configuration, whose incidence graph has large girth. In \cite{MSV19}, it is proved that for every $r$ there exist point-line configurations, whose incidence graph has $N$ vertices, has at least $\Omega(N^{1+\frac{4}{r^{2}+6r-3}})$ edges, and has girth at least $r+5$. We show that this can be improved by simply looking at a random subset of the standard point-line configuration. Most of this section is devoted to the proof of the following lemma.
	
	\begin{lemma}\label{lemma:girth}
		For every positive integer $N_{0}$, there exists a point-line configuration $(P,\mathcal{L})$ such that $|P|=|\mathcal{L}|=N_{0}$, $|I(P,\mathcal{L})|\geq N_{0}^{1+\frac{1}{2k-4}+o(1)}$, and the incidence graph of $(P,\mathcal{L})$ has girth at least $2k-2$. 
	\end{lemma}
	
	Let $N$ be a positive integer, which we specify later, and let $(P_{0},\mathcal{L}_{0})$ be the standard point-line configuration with $N$ points and $N$ lines. Let $B_{0}$ be the incidence graph of $(P_{0},\mathcal{L}_{0})$. Note that every vertex of $B_{0}$ has degree at most $N^{1/3}$, and at least $\frac{N}{2}$ lines and at least $\frac{N}{2}$ points have degree at least $\frac{N^{1/3}}{2}$. Let $H$ be the graph on $P_{0}$ in which $p$ and $p'$ are joined by an edge if $p,p'\in l$ for some $l\in \mathcal{L}_{0}$. Note that by the previous observation, the degree of each vertex of $H$ is at most $N^{2/3}$. 
	
	\begin{claim}\label{claim:common}
		Let $p,p'\in P_{0}$ be distinct vertices. The number of common neighbors of $p$ and $p'$ in $H$ is at most $N^{1/3+o(1)}$.
	\end{claim}
	
	\begin{proof}
		Translate our configuration such that $p=(0,0)$, let $p'=(a,b)$, and let $q=(u,v)\in P_{0}$ be a common neighbor of $p$ and $p'$.  Recall that the slope of every line $l\in \mathcal{L}_0$ is an integer.  Therefore, as $p$ and $q$ are on the same line of $\mathcal{L}_{0}$, we have $u|v$ and $|\frac{v}{u}|\leq N^{1/3}$. Also, as $p'$ and $q$ are on the same line of $\mathcal{L}_{0}$, we have $(u-a)|(v-b)$ and $|\frac{v-b}{u-a}|\leq N^{1/3}$. Set $t=\frac{v}{u} \in \mathbb{Z}$, which implies
		$$v-b=t(u-a)+ta-b.$$
		Then by dividing both sides by $(u-a)$, we can see that $(u-a)|(ta-b)$. Here, $ta-b=O(N)$, which means that the number of divisors (including both positive and negative) of $ta-b$ is $N^{o(1)}$ (see, e.g., \cite{HW}). Therefore, given $t$, $u$ can have at most $N^{o(1)}$ distinct values, and also $t$ can take at most $2N^{1/3}$ distinct values. Thus, $p$ and $p'$ have at most $N^{1/3+o(1)}$ common neighbors in $H$.
	\end{proof}
	
	Now we will use this claim to count the number of cycles of given length in $B_{0}$.
	
	\begin{claim}\label{claim:ncycles}
		For every $r\geq 3$, the number of cycles of length $2r$ in $B_{0}$ is at most $N^{2r/3+o(1)}$.
	\end{claim}
	
	\begin{proof}
		Note that each cycle of length $2r$ in $B_{0}$ gives rise to a unique cycle of length $r$ in $H$ (using that $B_{0}$ is $C_{4}$-free), so it is enough to count the cycles of length $r$ in $H$.
		
		Consider two cases. If $r$ is even, count the cycles $p_{1},\dots,p_{r}$ in $H$ as follows. We have at most $N^{r/2}$ ways to choose the vertices $p_{2i}$ for $i=1,\dots,\frac{r}{2}$. After  $p_{2i-2}$ and $p_{2i}$ are fixed (indices are meant modulo $r$), we have at most $N^{1/3+o(1)}$ choices for $p_{2i-1}$ by Claim \ref{claim:common}. Therefore, we have at most $N^{r/6+o(1)}$ further choices for the vertices $p_{2i-1}$ for $i=1,\dots,r/2$. In total, there are at most $N^{2r/3+o(1)}$ cycles of length $r$ in $H$.
		
		If $r$ is odd, we proceed similarly. We have at most $N^{(r-1)/2}$ ways to choose the vertices $p_{2i}$ for $i=1,\dots,\frac{r-1}{2}$. Then, we have at most $N^{2/3}$ further choices for $p_{1}$ as it is a neighbor of $p_{2}$. For each vertex $p_{2i-1}$, where $i=2,\dots,(r+1)/2$, there are at most $N^{1/3+o(1)}$ choices by Claim \ref{claim:common}, which gives at most $N^{(r-1)/6+o(1)}$ further choices. Thus, in total, there are at most $N^{2r/3+o(1)}$ cycles of length $r$ in $H$.	
	\end{proof}
	
	Let us remark that in general, the $o(1)$ error term is needed in Claim \ref{claim:ncycles}. Indeed, a result of Klav\'ik, Kr\'al' and Mach \cite{KKM11} gives that the number of cycles of length $6$ in $B_{0}$ is $\Omega(N^{2}\log\log N)$.
	
	\begin{proof}[Proof of Lemma \ref{lemma:girth}]
		For $3\leq r\leq k-2$, let $C_{r}$ be the number of cycles of length $2r$ in $B_{0}$. Recall that $(P_0,\mathcal{L}_0)$ is the standard point-line configuration with $N$ points and $N$ lines, and $B_0$ is its incidence graph.  In what follows, let $q\in (0,1)$, and let $P'$ and $\mathcal{L}'$ be subsets of $P_{0}$ and $\mathcal{L}_{0}$, in which each element is present independently with probability $q$, and let $B'$ be the subgraph of $B_{0}$ induced by $P'\cup\mathcal{L}'$. Also, let $X_{r}$ be the number of cycles of length $2r$ in $B'$. Then $\mathbb{E}(|P'|)=\mathbb{E}(|\mathcal{L}'|)=qN$ and $\mathbb{E}(X_{r})=q^{2r}C_{r}$.
		
		In what comes, we suppose that $q>n^{-1/3+\epsilon}$ for some $\epsilon>0$. Let $\mathcal{A}$ be the event that $B'$ satisfies the following properties:
		\begin{itemize}
			\item $\frac{qN}{2}<|P'|,|\mathcal{L}'|< 2qN$,
			\item the degree of every vertex of $B'$ is at most $2qN^{1/3}$,
			\item there are at least $\frac{qN}{4}$ points and lines in $B'$ whose degree in $B'$ is between $\frac{qN^{1/3}}{4}$ and $2qN^{1/3}$.
		\end{itemize}
		
		Then, by standard concentration inequalities (e.g. Chernoff's inequality) we have  $\mathbb{P}(\mathcal{A})>2/3$. If $\mathcal{A}$ holds, then we also have that $|E(B')|\geq \frac{q^{2}N^{4/3}}{16}$. Choose $q$ such that 
		$$\sum_{r=3}^{k-2}\mathbb{E}(X_{r})=\sum_{r=3}^{k-2}q^{2r}C_{r}\leq \frac{qN}{128}$$
		holds. As $C_{r}\leq N^{2r/3+o(1)}$ by Claim \ref{claim:ncycles}, we get that some $q= N^{-\frac{2k-7}{6k-15}+o(1)}$ suffices. But then by Markov's ineqaulity, with probability at least $\frac{1}{2}$, we have $X_{3}+\dots+X_{k-2}\leq \frac{qN}{64}.$ 
		Therefore, there exists a choice for $B'$ such that $\mathcal{A}$ holds and the number of cycles of length at most $2k-4$ in $B'$ is at most $\frac{qN}{64}.$ Delete an arbitrary vertex from each such cycle, and let $B''$ be the resulting graph with parts $P''$ and $\mathcal{L}''$. Then $\frac{qN}{4}<|P''|,|\mathcal{L}''|<2qN$, and we deleted at most $\frac{qN}{64}\cdot 2qN^{1/3}=\frac{q^{2}N^{4/3}}{32}$ edges. Therefore, $B''$ still has at least $\frac{q^{2}N^{4/3}}{32}$ edges, and the girth of $B''$ is at least $2k-2$. Define $N_0$ such that $\frac{qN}{8}<N_{0}<\frac{qN}{4}$, then $N_{0}=N^{\frac{4k-8}{6k-15}+o(1)}$. By sampling random $N_{0}$ element subsets of $P''$ and $\mathcal{L}''$, we get that there exists  an induced subgraph $B$ of $B''$ with parts $P$ and $\mathcal{L}$, each of size $N_{0}$, such that $B$ has at least $\frac{q^{2}N^{4/3}}{128}=N^{\frac{4k-6}{6k-15}+o(1)}=N_{0}^{1+\frac{1}{2k-4}+o(1)}$ edges and $B$ has girth at least $2k-2$. The point-line configuration $(P,\mathcal{L})$ satisfies our desired properties.
	\end{proof}
	
	Now we are ready to prove Theorem \ref{thm:girth}.
	
	\begin{proof}[Proof of Theorem \ref{thm:girth}]
		Let $(P,\mathcal{L})$ be a point-line configuration such that $|P|=|\mathcal{L}|=N_{0}$, $|I(P,\mathcal{L})|\geq N_{0}^{1+\frac{1}{2k-4}+o(1)}$, and the incidence graph of $(P,\mathcal{L})$ has girth at least $2k-2$. By Lemma \ref{lemma:girth}, such a configuration exists. Let $n=|I(P,\mathcal{L})|$, and construct the same graph $G$ as in Section \ref{sect:1} with respect to $(P,\mathcal{L})$. Then $G$ is a Hasse diagram on $n$ vertices, the independence number of $G$ is at most $2N_{0}=n^{1-\frac{1}{2k-3}+o(1)}$ by Claim \ref{claim:ind}, and the girth of $G$ is at least $k$ by Claim \ref{claim:girth}. Hence, $G$ has chromatic number at least $n^{\frac{1}{2k-3}+o(1)}$. 
	\end{proof}

	\section{Concluding remarks}
	
	We proved that the chromatic number of a Hasse diagram on $n$ vertices can be as large as $\Omega(n^{1/4})$. However, it would be interesting to decide whether  it can be $n^{1/2+o(1)}$, which is the general upper bound for all-triangle free  graphs. We suspect that the answer is no.
	
	\begin{conjecture}
		There exists $\epsilon>0$ such that if $G$ is a Hasse diagram on $n$ vertices, then $\chi(G)\leq n^{1/2-\epsilon}$.
	\end{conjecture}
	
	Claim \ref{claim:common} in Section 3 has several interesting applications in incidence geometry, some of which we describe below.  
	
	Given a point set $P$ and a set of lines $\mathcal{L}$ in the plane, we say that $(P,\mathcal{L})$ contains a \emph{$k\times k$ grid} if there are $k^2$ distinct points $p_{i,j} \in P$, where $(i,j) \in [k]\times [k]$, and two sets of distinct lines $l_1,\ldots, l_k \in \mathcal{L}$ and $l'_1,\ldots, l'_k\in \mathcal{L}$ such that $p_{i,j} \in l_i\cap l'_j$.  In \cite{MS}, it was shown that any set of $n$ points and $n$ lines in the plane that does not contain a $k\times k$-grid has at most $O(n^{4/3- \varepsilon})$ incidences, where $\varepsilon$ depends on $k$.  A straightforward adaptation of Lemma \ref{lemma:girth} gives the following.
	
	\begin{theorem}
		For every positive integer $k\geq 2$, there exists a set of $n$ points and $n$ lines in the plane that does not contain a $k\times k$ grid, and determines at least $n^{4/3 - \Theta(1/k)}$ incidences.
		
	\end{theorem}

	We say that  $(P,\mathcal{L})$ contains a \emph{$k$-fan} if there are $k+1$ distinct points $p_0,p_1,\ldots, p_k \in P$ and a set of $k+1$ distinct lines $l_0,l_1,\ldots, l_k \in \mathcal{L}$ such that $p_0 \in l_1\cap l_2\cap \cdots \cap l_k$, $p_i \in l_0 \cap l_i$ for $i = 1,\ldots, k$ and $p_0 \not\in l_0$.  In \cite{S}, Solymosi showed that every set of $n$ points and $n$ lines in the plane that does not contain a $2$-fan determines at most $o(n^{4/3})$ incidences.  However, for $k\geq 3$, the following conjecture remains open (see \cite{BMP}).
	
	\begin{conjecture}
		For fixed $k\geq 3$, every set of $n$ points and $n$ lines in the plane that does not contain a $k$-fan determines at most $o(n^{4/3})$ incidences.
		
	\end{conjecture}
	
	Again, a straightforward adaptation of Lemma \ref{lemma:girth} gives the following.

	\begin{theorem}
		There exists a set of $n$ points and $n$ lines in the plane that does not contain a 2-fan and determines at least $n^{7/6 + o(1)}$ incidences.  In general, for $k\geq 3$, there exists a set of $n$ points and $n$ lines in the plane that does not contain a $k$-fan and determines at least $n^{4/3 - \Theta(1/k)}$ incidences.
	\end{theorem}

	\section{Acknowledgements}

	We would like to thank Jacob Fox, Nitya Mani, Abhishek Methuku, J\'anos Pach, Prasanna Ramakrishnan, Benny Sudakov, and Adam Zsolt Wagner for valuable discussions.
	
	The second author was also partially supported by MIPT and the grant of Russian Government N 075-15-2019-1926.

\end{document}